\documentclass[oneside,english]{amsart}
\usepackage{lmodern}
\usepackage[T1]{fontenc}
\usepackage[latin9]{inputenc}
\usepackage{geometry}
\usepackage{color}
\usepackage{amsbsy}
\usepackage{amstext}
\usepackage{amsthm}
\usepackage{amssymb}

\makeatletter
\numberwithin{equation}{section}
\numberwithin{figure}{section}
\theoremstyle{plain}
\newtheorem{thm}{\protect\theoremname}
\theoremstyle{plain}
\newtheorem{prop}[thm]{\protect\propositionname}
\theoremstyle{plain}
\newtheorem{lem}[thm]{\protect\lemmaname}

\makeatother

\usepackage{babel}
\providecommand{\lemmaname}{Lemma}
\providecommand{\propositionname}{Proposition}
\providecommand{\theoremname}{Theorem}

\begin{document}
\global\long\def\e{e}%
\global\long\def\V{{\rm Vol}}%
\global\long\def\bs{\boldsymbol{\sigma}}%
\global\long\def\br{\boldsymbol{\rho}}%
\global\long\def\bp{\boldsymbol{\pi}}%
\global\long\def\btau{\boldsymbol{\tau}}%
\global\long\def\bx{\mathbf{x}}%
\global\long\def\by{\mathbf{y}}%
\global\long\def\bz{\mathbf{z}}%
\global\long\def\bv{\mathbf{v}}%
\global\long\def\bu{\mathbf{u}}%
\global\long\def\bi{\mathbf{i}}%
\global\long\def\bn{\mathbf{n}}%
\global\long\def\grad{\nabla_{sp}}%
\global\long\def\Hess{\nabla_{sp}^{2}}%
\global\long\def\lp{\Delta_{sp}}%
\global\long\def\gradE{\nabla_{\text{Euc}}}%
\global\long\def\HessE{\nabla_{\text{Euc}}^{2}}%
\global\long\def\HessEN{\hat{\nabla}_{\text{Euc}}^{2}}%
\global\long\def\ddq{\frac{d}{dR}}%
\global\long\def\qs{q_{\star}}%
\global\long\def\qss{q_{\star\star}}%
\global\long\def\lm{\lambda_{min}}%
\global\long\def\Es{E_{\star}}%
\global\long\def\As{A_{\star}}%
\global\long\def\EH{E_{\Hess}}%
\global\long\def\Esh{\hat{E}_{\star}}%
\global\long\def\ds{d_{\star}}%
\global\long\def\Cs{\mathscr{C}_{\star}}%
\global\long\def\nh{\boldsymbol{\hat{\mathbf{n}}}}%
\global\long\def\BN{\mathbb{B}^{N}}%
\global\long\def\ii{\mathbf{i}}%
\global\long\def\SN{\mathbb{S}^{N-1}}%
\global\long\def\SM{\mathbb{S}^{M-1}}%
\global\long\def\SNq{\mathbb{S}^{N-1}(q)}%
\global\long\def\SNqd{\mathbb{S}^{N-1}(q_{d})}%
\global\long\def\SNqp{\mathbb{S}^{N-1}(q_{P})}%
\global\long\def\nd{\nu^{(\delta)}}%
\global\long\def\nz{\nu^{(0)}}%
\global\long\def\cls{c_{LS}}%
\global\long\def\qls{q_{LS}}%
\global\long\def\dls{\delta_{LS}}%
\global\long\def\E{\mathbb{E}}%
\global\long\def\P{\mathbb{P}}%
\global\long\def\R{\mathbb{R}}%
\global\long\def\spp{{\rm Supp}(\mu_{P})}%
\global\long\def\indic{\mathbf{1}}%
\global\long\def\lsc{\mu_{{\rm sc}}}%
\newcommand{\SNarg}[1]{\mathbb S^{N-1}(#1)}
\global\long\def\se{s(E)}%
\global\long\def\ses{s(\Es)}%
\global\long\def\so{s(0)}%
\global\long\def\sef{s(E_{f})}%
\global\long\def\seinf{s(E_{\infty})}%
\global\long\def\L{\mathcal{L}}%
\global\long\def\gflow#1#2{\varphi_{#2}(#1)}%
\global\long\def\S{\mathscr{S}}%
\global\long\def\Asm{{\rm (A)}}%
\global\long\def\Asmnb{{\rm A}}%
\global\long\def\Frep{F^{{\rm Rep}}}%
\global\long\def\s{\mathfrak{s}}%
\global\long\def\e{e}%
\global\long\def\EsN{E_{\star,N}}%

\theoremstyle{definition}
\newtheorem*{assumption}{Assumption}
\title{Convergence of the free energy for spherical spin glasses}

\author{Eliran Subag}
\address{\tiny{Eliran Subag, incumbent of the Skirball Chair for New Scientists, Department of Mathematics, Weizmann Institute of Science, Rehovot 76100, Israel.}}
\email{eliran.subag@weizmann.ac.il}

\thanks{This project has received funding from the Israel Science Foundation (Grant Agreement No. 2055/21). }

\begin{abstract}
We prove that the free energy of any spherical mixed $p$-spin model
converges as the dimension $N$ tends to infinity. While the convergence
is a consequence of the Parisi formula, the proof we give is independent
of the formula and uses the well-known Guerra-Toninelli interpolation
method. The latter was invented for models with Ising spins to prove
that the free energy is super-additive and therefore (normalized by
$N$) converges. In the spherical case, however, the configuration
space is not a product space and the interpolation cannot be applied
directly. We first relate the free energy on the sphere of dimension
$N+M$ to a free energy defined on the product of spheres in dimensions
$N$ and $M$ to which we then apply the interpolation method. This
yields an approximate super-additivity which is sufficient to prove
the convergence. 
\end{abstract}

\maketitle

\section{Introduction}

In this work we consider the spherical mixed $p$-spin spin glass
models. The limit of their free energy as the dimension tends to infinity
is given by the celebrated Parisi formula \cite{ParisiFormula,Parisi1980}
or its representation by Crisanti and Sommers \cite{Crisanti1992}.
The formula was proved by Talagrand \cite{Talag} after a breakthrough
by Guerra \cite{GuerraBound} for models with even interactions and
later extended to general mixtures by Chen \cite{Chen}, using the
Aizenman-Sims-Starr representation \cite{ASSscheme} and ultrametricity
\cite{MPSTV2,MPSTV1,ultramet}.

Our goal in this note is only to prove the convergence of the free
energy, but without relying on the heavy machinery which was developed
to prove the Parisi formula. One of our main motivations comes from
recent works \cite{TAPChenPanchenkoSubag,TAPIIChenPanchenkoSubag,FElandscape,SubagTAPpspin}
on the Thouless-Anderson-Palmer (TAP) approach \cite{TAP}. In \cite{FElandscape}
we proved a generalized TAP representation for the free energy of
the spherical models and that for the maximal multi-samplable overlap
the correction term in the representation coincides with the classical
Onsager correction. Importantly, the proof of those results was independent
of the Parisi formula, but for the latter result on the Onsager correction
we had to assume that the free energy converges.\footnote{The statement of the TAP representation in \cite{FElandscape} also
assumes the convergence of the free energy, as it relates the $N\to\infty$
limit of several quantities including the free energy itself. An equivalent
way to phrase the representation is that the difference of the pre-limits
of both sides of it tends to zero as $N\to\infty$. This statement
can be proved by a slight modification of the existing proof in \cite{FElandscape}
and it does not require each of the terms, and the free energy in particular,
to have a limit.} In another work \cite{SubagTAPpspin}, we used the TAP representation
to compute the free energy of the spherical pure $p$-spin models
from the generalized TAP representation of \cite{FElandscape}, also
there assuming the convergence of the free energy. Our main result
in the current paper fills the gap and removes those assumptions from
\cite{FElandscape,SubagTAPpspin} without appealing to the Parisi
formula. 

For models with Ising spins, defined on the hyper-cube $\Sigma_{N}=\{\pm1\}^{N}$,
the convergence of the free energy was proved by Guerra and Toninelli
\cite{GuerraToninelli} who invented a simple, yet ingenious, interpolation
technique to show that the (unnormalized) free energy is superadditive.
The convergence immediately follows from superadditivity by invoking
Fekete's Lemma. The argument of \cite{GuerraToninelli} exploits the
fact that the configuration space $\Sigma_{N+M}=\Sigma_{N}\times\Sigma_{M}$
in dimension $N+M$ is equal to the product of the configuration space
in dimensions $N$ and $M$. In the spherical setting, this is no
longer the case, and the method of \cite{GuerraToninelli} cannot
be adapted directly. In our proof we therefore first relate the free
energy in dimension $N+M$ to another free energy defined on the product
of the spherical configuration space in dimensions $N$ and $M$ (using
the Hamiltonian in dimension $N+M$), to which we will be able to
apply the Guerra-Toninelli interpolation technique.

The spherical mixed $p$-spin spin glass model is defined as follows.
Suppose that $\gamma_{p}\geq0$ is a sequence such that $\sum_{p=1}^{\infty}\gamma_{p}^{2}(1+\epsilon)^{p}<\infty$
for small enough $\epsilon>0$. The mixed $p$-spin \emph{Hamiltonian
}corresponding to the \emph{mixture} $\xi(t)=\sum_{p\geq1}\gamma_{p}^{2}t^{p}$
is the random function on the sphere 
\begin{equation}
S_{N}:=\{\bs=(\sigma_{1},\ldots,\sigma_{N})\in\R^{N}:\,\|\bs\|=\sqrt{N}\},\label{eq:SN}
\end{equation}
given by 
\begin{equation}
H_{N}(\bs)=\sum_{p=1}^{\infty}\gamma_{p}N^{-\frac{p-1}{2}}\sum_{i_{1},\dots,i_{p}=1}^{N}J_{i_{1},\dots,i_{p}}\sigma_{i_{1}}\cdots\sigma_{i_{p}},\label{eq:Hamiltonian}
\end{equation}
where $J_{i_{1},\dots,i_{p}}$ are i.i.d. standard normal variables.
An easy calculation shows that the covariance function of the centered
Gaussian field $H_{N}(\bs)$ is 
\[
\E H_{N}(\bs)H_{N}(\bs')=N\xi(R(\bs,\bs')),
\]
where $R(\bs,\bs'):=\frac{1}{N}\bs\cdot\bs':=\frac{1}{N}\sum_{i\leq N}\sigma_{i}\sigma_{i}'$
is called the overlap of $\bs$ and $\bs'$.

The free energy is defined by 
\begin{equation}
F_{N}:=\frac{1}{N}\E\log\int_{S_{N}}e^{H_{N}(\bs)}d\mu_{N}(\bs),\label{eq:Fbeta-1}
\end{equation}
where $\mu_{N}$ is the uniform measure on $S_{N}$. The following
is our main result. 
\begin{thm}
\label{thm}$F_{N}$ converges as $N\to\infty$.
\end{thm}

At the very last step of the proof of Theorem \ref{thm}, after we
apply the Guerra-Toninelli interpolation, we will need to invoke Talagrand's
positivity principle \cite{TalagBook2003} in order to restrict to
overlap values in $[0,1]$ (on which $\xi(t)$ is convex). The positivity
principle applies to general mixtures, if we add a certain perturbation
to the Hamiltonian.\footnote{For mixtures such that $\xi(t)$ is convex on $[-1,1]$ this step
is not required and one can work with the original Hamiltonian without
adding a perturbation.} We will give the precise definition of the perturbed Hamiltonian
$\bar{H}_{N}(\bs)$ and its associated free energy $\bar{F}_{N}$,
which satisfies
\begin{equation}
\lim_{N\to\infty}|F_{N}-\bar{F}_{N}|=0,\label{eq:Fbd}
\end{equation}
in a moment. Before that, we state the following approximate superadditivity
of $\bar{F}_{N}$ and observe how the convergence of $F_{N}$ follows
from it. 
\begin{prop}
\label{prop}For any $N$ and $M$, 
\[
(N+M)\bar{F}_{N+M}\geq N\bar{F}_{N}+M\bar{F}_{M}-C_{N,M},
\]
for some numbers $C_{N,M}$ such that 
\[
\limsup_{M\to\infty}\limsup_{N\to\infty}\frac{C_{N,M}}{M}=0.
\]
\end{prop}

\begin{proof}
[Proof of Theorem \ref{thm}] In light of (\ref{eq:Fbd}), it is enough
to show that $\bar{F}_{N}$ converges. By induction on $k$, for any
$N'$, $M$ and $k$,
\[
(N'+kM)\bar{F}_{N'+kM}\geq N'\bar{F}_{N'}+kM\bar{F}_{M}-\sum_{i=0}^{k-1}C_{N'+iM,M}.
\]

Let $\delta>0$ be an arbitrary number. Choose some large $M$ such
that 
\[
\bar{F}_{M}>\limsup_{N\to\infty}\bar{F}_{N}-\delta\quad\text{and}\quad\limsup_{N\to\infty}\frac{C_{N,M}}{M}<\delta.
\]

Given some $N$, let $N'\in\{0,1,\ldots,M-1\}$ and $k\geq0$ be the
integers such that $N=N'+kM$. Then by dividing both sides of the
inequality above by $N$ and taking limits we obtain that with $M$
fixed,
\[
\liminf_{N\to\infty}\bar{F}_{N}\geq\bar{F}_{M}-\limsup_{N\to\infty}\frac{C_{N,M}}{M}\geq\limsup_{N\to\infty}\bar{F}_{N}-2\delta.
\]
Since $\delta>0$ is arbitrary, $\bar{F}_{N}$ converges and the theorem
follows.
\end{proof}
We now turn to the definition of the perturbed Hamiltonian $\bar{H}_{N}(\bs)$,
which we take from Section 3.2 of \cite{PanchenkoBook}. Let $H_{N,p}(\bs)$
denote the pure $p$-spin Hamiltonian with mixture $\xi(t)=t^{p}$.
For $p\geq1$, let $g_{N,p}(\bs)$ be a sequence of Hamiltonians such
that $g_{N,p}(\bs)=\frac{1}{\sqrt{N}}H_{N,p}(\bs)$ in distribution.
Let $x_{p}$ be i.i.d. random variables uniform on $[1,2]$. Assume
that $g_{N,p}(\bs)$ and $x_{p}$ are independent of each other and
everything else. Set $s_{N}=N^{c}$ for some $c\in(1/4,1/2)$, which
we now fix once and for all. Finally, define
\begin{equation}
g_{N}(\bs)=\sum_{p=1}^{\infty}2^{-p}x_{p}g_{N,p}(\bs)\quad\text{and}\quad\bar{H}_{N}(\bs)=H_{N}(\bs)+s_{N}g_{N}(\bs).\label{eq:Hbar}
\end{equation}
We define the free energy $\bar{F}_{N}$ from Proposition \ref{prop}
by
\begin{equation}
\bar{F}_{N}:=\frac{1}{N}\E\log\int_{S_{N}}e^{\bar{H}_{N}(\bs)}d\mu_{N}(\bs),\label{eq:Ftilde}
\end{equation}
where the expectation is also w.r.t. the randomness of the uniform
variables $x_{p}$. The choice of $s_{N}$ as above implies (\ref{eq:Fbd}),
see \cite{PanchenkoBook}. 

The proof of Proposition \ref{prop} will consist of two steps, stated
in the lemmas below. The first will be to relate the free energy in
dimension $N+M$ to another free energy defined on the product space
$S_{N}\times S_{M}\subset S_{N+M}$. By an abuse of notation we write
$\bar{H}_{N+M}(\br,\btau)$ for $\bar{H}_{N+M}((\br,\btau))$ where
$(\br,\btau)\in S_{N+M}$ denotes the vector obtained by concatenating
$\br\in S_{N}$ and $\btau\in S_{M}$.
\begin{lem}
\label{lem:coarea}For any $N$ and $M$, 
\begin{equation}
(N+M)\bar{F}_{N+M}\geq\E\log\int_{S_{N}\times S_{M}}e^{\bar{H}_{N+M}(\br,\btau)}d\mu_{N}\times\mu_{M}(\br,\btau)-C_{N,M}',\label{eq:lemcoarea}
\end{equation}
for some numbers $C_{N,M}'$ such that 
\[
\limsup_{M\to\infty}\limsup_{N\to\infty}\frac{C_{N,M}'}{\sqrt{M}}=C(\xi)
\]
where $C(\xi)$ is a constant which depends only on $\xi$. 
\end{lem}

We remark that the perturbation has no role in the proof the lemma
above, and it still holds also if we work with the unperturbed Hamiltonian
$H_{N}$ and free energy $F_{N}$. The free energy in the right-hand
side of (\ref{eq:lemcoarea}) is defined on a product space. We will
therefore be able to apply to it the Guerra-Toninelli interpolation
and obtain the following lemma. Its proof is where we will use Talagrand's
positivity principle.
\begin{lem}
\label{lem:subadd}For any $N$ and $M$, 
\begin{equation}
\E\log\int_{S_{N}\times S_{M}}e^{\bar{H}_{N+M}(\br,\btau)}d\mu_{N}\times\mu_{M}(\br,\btau)\geq N\bar{F}_{N}+M\bar{F}_{M}-C_{N,M},\label{eq:NM}
\end{equation}
for some numbers $C_{N,M}$ as in Proposition \ref{prop}.
\end{lem}

Proposition \ref{prop} directly follows from the two lemmas. The
rest of the paper consists of the proof of Lemmas \ref{lem:coarea} and \ref{lem:subadd} in Sections \ref{sec:lemCoarea} 
and \ref{sec:lemsubadd}, respectively.

\section{\label{sec:lemCoarea}Proof of Lemma \ref{lem:coarea}}

Fix some integers $N,M\geq1$. We will denote by $\nu_{d}$ the $d-1$
dimensional Hausdorff measure. In principle, it depends on the dimension
of the ambient space, but we will omit this from the notation. By
definition, 
\begin{equation}
(N+M)\bar{F}_{N+M}=\E\log\left(\frac{1}{\nu_{N+M}(S_{N+M})}\int_{S_{N+M}}e^{\bar{H}_{N+M}(\bs)}d\nu_{N+M}(\bs)\right).\label{eq:int1}
\end{equation}

Similarly to (\ref{eq:SN}), we will use the notation 
\[
S_{N}(r):=\{\bs=(\sigma_{1},\ldots,\sigma_{N})\in\R^{N}:\,\|\bs\|=r\}
\]
for the sphere of radius $r$ in $\R^{N}$.\textbf{\textcolor{red}{{}
}}Define the function 
\[
\eta(r)=\sqrt{N+M-r^{2}}.
\]
For two vectors $\br\in\R^{N}$ and $\btau\in\R^{M}$, we will denote
by $(\br,\btau)\in\R^{N+M}$ the vector obtained by concatenating
$\br$ and $\btau$. Given some $\bs\in S_{N+M}$, whenever we write
$\bs=(\br,\btau)$ it should be understood that $\br$ and $\btau$
are the projections of $\bs$ to $\R^{N}$ and $\R^{M}$.

Consider the mapping $\phi:S_{N+M}\to\R$, $\bs=(\br,\btau)\mapsto\|\btau\|$
and note that
\[
\phi^{-1}(r)=S_{N}(\eta(r))\times S_{M}(r).
\]
For $\bs=(\br,\btau)\in S_{N+M}$ such that $0<\|\btau\|<\sqrt{N+M}$,
the Jacobian of the differential $D_{\bs}\phi:T_{\bs}S_{N+M}\to T_{\|\btau\|}\R$
is equal to\footnote{Where on $S_{N+M}$ and $\phi^{-1}(r)$ we assume the Riemannian structure
induced by $\R^{N+M}$ which is compatible with the Hausdorff measures
on them.}
\[
\sqrt{\frac{N+M-\|\btau\|^{2}}{N+M}}=\frac{\eta(\|\btau\|)}{\sqrt{N+M}}.
\]

For an interval $I\subset[0,\sqrt{N+M}]$, consider the sub-manifold
\begin{align*}
S_{N+M}^{I} & :=\phi^{-1}(I)=\left\{ \bs=(\br,\btau)\in S_{N+M}:\,\|\btau\|\in I\right\} .
\end{align*}
By the coarea formula,
\[
\int_{S_{N+M}^{I}}e^{\bar{H}_{N+M}(\bs)}d\nu_{N+M}(\bs)=\int_{I}\frac{\sqrt{N+M}}{\eta(r)}X_{N,M}(r)dr,
\]
where we define

\[
X_{N,M}(r):=\int_{S_{N}(\eta(r))\times S_{M}(r)}e^{\bar{H}_{N+M}(\br,\btau)}d\nu_{N}\times\nu_{M}(\br,\btau).
\]

Applying the same argument to the constant function identically equal
to $1$ over $S_{N+M}$, we have that 
\begin{equation}
\begin{aligned} & \nu_{N+M}(S_{N+M}^{I})=\int_{S_{N+M}^{I}}d\nu_{N+M}(\bs)\\
 & =\int_{I}\frac{\sqrt{N+M}}{\eta(r)}\left(\frac{\eta(r)}{\sqrt{N}}\right)^{N-1}\left(\frac{r}{\sqrt{M}}\right)^{M-1}\nu_{N}(S_{N})\nu_{M}(S_{M})dr.
\end{aligned}
\label{eq:int2-2-1}
\end{equation}

Using the above and (\ref{eq:int1}), we have that
\begin{align}
 & (N+M)\bar{F}_{N+M}\geq\E\log\left(\frac{\nu_{N+M}(S_{N+M}^{I})}{\nu_{N+M}(S_{N+M})}\frac{\int_{S_{N+M}^{I}}e^{\bar{H}_{N+M}(\bs)}d\nu_{N+M}(\bs)}{\nu_{N+M}(S_{N+M}^{I})}\right)\nonumber \\
 & =\log\frac{\nu_{N+M}(S_{N+M}^{I})}{\nu_{N+M}(S_{N+M})}+\E\log\frac{\int_{I}\frac{\sqrt{N+M}}{\eta(r)}X_{N,M}(r)dr}{\nu_{N+M}(S_{N+M}^{I})}.\label{eq:E}
\end{align}

From now on, we will work with the interval 
\begin{equation}
I=I^{+}\cup I^{-}:=[\sqrt{M}-a,\sqrt{M}]\cup[\sqrt{M},\sqrt{M}+a],\label{eq:I}
\end{equation}
where $a\in(0,1)$ is some fixed number, the value of which will not
be important. We will show that the first term in (\ref{eq:E}) converges
to a constant and prove a lower bound for the second term, as $N\to\infty$
first and then $M\to\infty$.

The following quite elementary lemma is usually attributed to Poincar\'{e}
\cite{Poincare}.
\begin{lem}
[Poincar\'{e}'s limit]\label{lem:Poincare} Fix an integer $M\geq1$.
Suppose that $\bs=(\br,\btau)$ is a random point uniformly distributed
on $S_{N+M}$. As $N\to\infty$, the marginal distribution of $\btau$
weakly converges to the standard Gaussian distribution on $\R^{M}$. 
\end{lem}

Let $W_{M}\in\R^{M}$ be a random vector of i.i.d. standard Gaussian
variables. By the lemma,
\begin{align*}
 & \lim_{N\to\infty}\frac{\nu_{N+M}(S_{N+M}^{I})}{\nu_{N+M}(S_{N+M})}=\lim_{N\to\infty}\mu_{N+M}(S_{N+M}^{I})\\
 & =\P\left(\|W_{M}\|^{2}-M-a^{2}\in[-2a\sqrt{M},2a\sqrt{M}]\right).
\end{align*}
By the central limit theorem, for some constant $b>0$,
\begin{equation}
\lim_{M\to\infty}\lim_{N\to\infty}\frac{\nu_{N+M}(S_{N+M}^{I})}{\nu_{N+M}(S_{N+M})}=\P\left(W_{1}^{2}\in[-\sqrt{2}a,\sqrt{2}a]\right)=b.\label{eq:aI}
\end{equation}
Similarly, for the intervals $I^{+}$ and $I^{-}$ as defined in (\ref{eq:I})
we have that
\begin{equation}
\lim_{M\to\infty}\lim_{N\to\infty}\frac{\nu_{N+M}(S_{N+M}^{I^{\pm}})}{\nu_{N+M}(S_{N+M})}=\frac{b}{2}.\label{eq:aII}
\end{equation}

Define the two sets 
\[
D_{N,M}^{\pm}=\left\{ (\br,\btau)\in S_{N}\times S_{M}:\,\pm\frac{d}{dt}\bar{H}_{N+M}\Big(\br,\btau+t\frac{\btau}{\|\btau\|}\Big)\geq0\right\} .
\]
For $r\in(0,\sqrt{N+M})$, consider the mapping 
\begin{equation}
f_{r}:\,\begin{aligned}S_{N}(\sqrt{N})\times S_{M}(\sqrt{M}) & \to S_{N}(\eta(r))\times S_{M}(r),\\
(\br,\btau) & \mapsto\Big(\frac{\eta(r)}{\sqrt{N}}\br,\frac{r}{\sqrt{M}}\btau\Big).
\end{aligned}
\label{eq:fr}
\end{equation}
Define the subsets
\[
D_{N,M}^{\pm}(r):=f_{r}(D_{N,M}^{\pm})\subset S_{N}(\eta(r))\times S_{M}(r)
\]
and variables 
\[
X_{N,M}^{\pm}(r):=\int_{D_{N,M}^{\pm}(r)}e^{\bar{H}_{N+M}(\br,\btau)}d\nu_{N}\times\nu_{M}(\br,\btau)
\]
and 
\[
Y_{N,M}^{\pm}:=\frac{\int_{I^{\pm}}\frac{\sqrt{N+M}}{\eta(r)}X_{N,M}^{\pm}(r)dr}{\nu_{N+M}(S_{N+M}^{I^{\pm}})}.
\]

To lower bound the second term in (\ref{eq:E}), we will prove a lower
bound for $Y_{N,M}^{+}\vee Y_{N,M}^{-}$ (where we denote by $a\vee b$
the maximum of $a,b\in\R$). The main estimates we will use are in
the following lemma which we prove below. 
\begin{lem}
\label{lem:estimate}There exist some positive constants $A=A(\xi)$
and $B=B(\xi)$ depending only on $\xi$ and random variables $L^{(1)}=L_{N,M}^{(1)}$
and $L^{(2)}=L_{N,M}^{(2)}$ such that:
\begin{enumerate}
\item \label{enu:pt1}For any $t>A$,
\[
\P\left(L^{(i)}\geq(N+M)^{\frac{2-i}{2}}t\right)\leq\exp\left(-\frac{N+M}{B}(t-A)^{2}\right).
\]
\item \label{enu:pt2}For any $r>\sqrt{M}$, 
\[
X_{N,M}^{+}(r)\geq\left(\frac{\eta(r)}{\sqrt{N}}\right)^{N-1}\left(\frac{r}{\sqrt{M}}\right)^{M-1}X_{N,M}^{+}(\sqrt{M})e^{-L^{(1)}|\eta(r)-\sqrt{N}|-L^{(2)}|r-\sqrt{M}|^{2}}.
\]
\item \label{enu:pt3}For any $r<\sqrt{M}$,
\[
X_{N,M}^{-}(r)\geq\left(\frac{\eta(r)}{\sqrt{N}}\right)^{N-1}\left(\frac{r}{\sqrt{M}}\right)^{M-1}X_{N,M}^{-}(\sqrt{M})e^{-L^{(1)}|\eta(r)-\sqrt{N}|-L^{(2)}|r-\sqrt{M}|^{2}}.
\]
\end{enumerate}
\end{lem}

For large $N$, using that for fixed $x\in\R$, $\sqrt{N+x}=\sqrt{N}+\frac{x}{2\sqrt{N}}+O(N^{-3/2})$,
one obtains that
\begin{align*}
\sup_{r\in I}\left(L^{(1)}|\eta(r)-\sqrt{N}|+L^{(2)}|r-\sqrt{M}|^{2}\right) & \leq\frac{a^{2}+2a\sqrt{M}}{\sqrt{N}}L^{(1)}+a^{2}L^{(2)}\\
 & \leq3\sqrt{\frac{M}{N}}L^{(1)}+L^{(2)}.
\end{align*}

Denote the ratio from (\ref{eq:E}) by
\[
T_{N,M}:=\frac{\int_{I}\frac{\sqrt{N+M}}{\eta(r)}X_{N,M}(r)dr}{\nu_{N+M}(S_{N+M}^{I})}.
\]
Since $X_{N,M}(r)\geq X_{N,M}^{\pm}(r)$, using (\ref{eq:aI}), (\ref{eq:aII})
and (\ref{eq:int2-2-1}) we have that 
\[
T_{N,M}\geq\frac{1}{4}Y_{N,M}^{\pm},
\]
assuming that $M\geq M_{0}(\xi)$ and $N\geq N_{0}(M,\xi)$, for appropriate
constants $M_{0}(\xi)$ and $N_{0}(M,\xi)$. 

On the event that $X_{N,M}^{+}(\sqrt{M})\geq\frac{1}{2}X_{N,M}(\sqrt{M})$,
using Part \ref{enu:pt2} of the lemma,
\[
Y_{N,M}^{+}\geq\frac{1}{2}\frac{X_{N,M}(\sqrt{M})}{\nu_{N}(S_{N})\nu_{M}(S_{M})}\exp\left(-3\sqrt{\frac{M}{N}}L^{(1)}-L^{(2)}\right).
\]
Respectively, using Part \ref{enu:pt3} of the lemma, on the event
that $X_{N,M}^{-}(\sqrt{M})\geq\frac{1}{2}X_{N,M}(\sqrt{M})$, the
same bound holds for $Y_{N,M}^{-}$.

Since $X_{N,M}(\sqrt{M})=X_{N,M}^{+}(\sqrt{M})+X_{N,M}^{-}(\sqrt{M})$,
deterministically, 
\[
X_{N,M}^{+}(\sqrt{M})\vee X_{N,M}^{-}(\sqrt{M})\geq\frac{1}{2}X_{N,M}(\sqrt{M})
\]
and 
\[
T_{N,M}\geq\frac{1}{8}\frac{X_{N,M}(\sqrt{M})}{\nu_{N}(S_{N})\nu_{M}(S_{M})}\exp\left(-3\sqrt{\frac{M}{N}}L^{(1)}-L^{(2)}\right).
\]

For $N$ and $M$ as above, we therefore have that 
\[
\E\log T_{N,M}\geq-\log8+\E\log\frac{X_{N,M}(\sqrt{M})}{\nu_{N}(S_{N})\nu_{M}(S_{M})}+\E\left(-3\sqrt{\frac{M}{N}}L^{(1)}-L^{(2)}\right).
\]
Note that the middle term above is equal to 
\[
\E\log\int_{S_{N}\times S_{M}}e^{\bar{H}_{N+M}(\br,\btau)}d\mu_{N}\times\mu_{M}(\br,\btau).
\]
Hence, by combining the above with (\ref{eq:E}) and (\ref{eq:aI}),
to complete the proof it remains to show that
\begin{equation}
\limsup_{M\to\infty}\limsup_{N\to\infty}\frac{1}{\sqrt{M}}\E\left(3\sqrt{\frac{M}{N}}L^{(1)}+L^{(2)}\right)\leq C(\xi),\label{eq:limsups}
\end{equation}
for some constant $C(\xi)$.

From Part \ref{enu:pt1} of Lemma \ref{lem:estimate} and the tail
formula,
\begin{align*}
\E L^{(1)} & \leq\sqrt{N+M}A+\int_{\sqrt{N+M}A}^{\infty}\exp\left(-\frac{N+M}{B}(\frac{t}{\sqrt{N+M}}-A)^{2}\right)dt\\
 & =\sqrt{N+M}A+\frac{\sqrt{\pi B}}{2}.
\end{align*}
Similarly,
\begin{align*}
\E L^{(2)} & \leq A+\frac{1}{2}\sqrt{\frac{\pi B}{N+M}}.
\end{align*}
This proves (\ref{eq:limsups}) and completes the proof. It remains
to prove Lemma \ref{lem:estimate}.

\subsection{Proof of Lemma \ref{lem:estimate}}

Let $L^{(1)}=L_{N,M}^{(1)}$ be the Lipschitz constant of $\bar{H}_{N+M}(\bs)$
over the ball of radius $\sqrt{N+M}$, 
\[
L^{(1)}:=\max_{\|\bs\|\leq\sqrt{N+M}}\|\nabla\bar{H}_{N+M}(\bs)\|=\max_{\|u\|=1}\max_{\|\bs\|\leq\sqrt{N+M}}u\cdot\nabla\bar{H}_{N+M}(\bs).
\]
Let $L^{(2)}=L_{N,M}^{(2)}$ be the maximal directional second order
derivative of $\bar{H}_{N+M}(\bs)$ over the same ball, 
\[
L^{(2)}:=\max_{\|u\|=1}\max_{\|\bs\|\leq\sqrt{N+M}}u^{T}\big(\nabla^{2}\bar{H}_{N+M}(\bs)\big)u.
\]

Suppose that $(\br,\btau)\in D_{N,M}^{+}$ and let $r>\sqrt{M}$.
Since
\begin{align*}
\|\frac{r}{\sqrt{M}}\btau-\btau\| & =|r-\sqrt{M}|
\end{align*}
and
\[
\|\frac{\eta(r)}{\sqrt{N}}\br-\br\|=|\eta(r)-\sqrt{N}|,
\]
by Taylor's approximation, 
\begin{align*}
\bar{H}_{N+M}(\br,\frac{r}{\sqrt{M}}\btau) & \geq\bar{H}_{N+M}(\br,\btau)-L^{(2)}|r-\sqrt{M}|^{2}
\end{align*}
and
\[
\bar{H}_{N+M}(\frac{\eta(r)}{\sqrt{N}}\br,\frac{r}{\sqrt{M}}\btau)\geq\bar{H}_{N+M}(\br,\frac{r}{\sqrt{M}}\btau)-L^{(1)}|\eta(r)-\sqrt{N}|.
\]

Hence,
\begin{align*}
 & \bar{H}_{N+M}\big(f(\br,\btau)\big)=\bar{H}_{N+M}(\frac{\eta(r)}{\sqrt{N}}\br,\frac{r}{\sqrt{M}}\btau)\\
 & \geq\bar{H}_{N+M}(\br,\btau)-L^{(1)}|\eta(r)-\sqrt{N}|-L^{(2)}|r-\sqrt{M}|^{2}.
\end{align*}

Clearly,
\[
\frac{\nu_{N}\big(D_{N,M}^{+}(r)\big)}{\nu_{N}\big(D_{N,M}^{+}(\sqrt{M})\big)}=\left(\frac{\eta(r)}{\sqrt{N}}\right)^{N-1}\left(\frac{r}{\sqrt{M}}\right)^{M-1}.
\]

Therefore,
\begin{align*}
 & X_{N,M}^{+}(r)=\int_{D_{N,M}^{+}(r)}e^{\bar{H}_{N+M}(\br,\btau)}d\nu_{N}\times\nu_{M}(\br,\btau)\\
 & =\frac{\nu_{N}\big(D_{N,M}^{+}(r)\big)}{\nu_{N}\big(D_{N,M}^{+}(\sqrt{M})\big)}\int_{D_{N,M}^{+}(\sqrt{M})}e^{\bar{H}_{N+M}\big(f(\br,\btau)\big)}d\nu_{N}\times\nu_{M}(\br,\btau)\\
 & \geq\left(\frac{\eta(r)}{\sqrt{N}}\right)^{N-1}\left(\frac{r}{\sqrt{M}}\right)^{M-1}X_{N,M}^{+}(\sqrt{M})e^{-L^{(1)}|\eta(r)-\sqrt{N}|-L^{(2)}|r-\sqrt{M}|^{2}}.
\end{align*}
This proves Part \ref{enu:pt2} of the lemma. Part \ref{enu:pt3}
follows by a similar argument.

For the rest of the proof we will work conditional on the uniform
random variables $x_{p}$. We will prove the bounds in Part \ref{enu:pt1}
with some constants $A=A(\xi)$ and $B=B(\xi)$ independent of the
values of $x_{p}$, which of course gives the same bounds unconditionally.
Note that under the conditioning, $\bar{H}_{N+M}(\bs)$ is a Gaussian
process. 

From the proof of \cite[Lemma 58]{geometryMixed}, one can see that
for some constant $A=A(\xi)$ that only depends on $\xi$, 
\begin{align*}
\E L^{(1)} & \leq\sqrt{N+M}A,\\
\E L^{(2)} & \leq A.
\end{align*}

Note that 
\[
\begin{aligned}\E(u\cdot\nabla\bar{H}_{N+M}(\bs))^{2} & =\frac{d}{dt}\Big|_{t=0}\frac{d}{ds}\Big|_{s=0}\E\left(\bar{H}_{N+M}(\bs+tu)\bar{H}_{N+M}(\bs+su)\right)\\
 & =(N+M)\frac{d}{dt}\Big|_{t=0}\frac{d}{ds}\Big|_{s=0}\xi_{N+M}^{x}\left(R(\bs+tu,\bs+su)\right),
\end{aligned}
\]
where, recalling the definition (\ref{eq:Hbar}), we define 
\[
\xi_{N}^{x}(t)=\xi(t)+\frac{s_{N}^{2}}{N}\sum_{p=1}^{\infty}4^{-p}x_{p}^{2}t^{p}.
\]
From this one can easily check that for any $u$ and $\bs$ as above,
\[
\E(u\cdot\nabla\bar{H}_{N+M}(\bs))^{2}\leq B,
\]
for some constant $B=B(\xi)$. By a similar argument, for such $u$
and $\bs$,
\[
\E(u^{T}\big(\nabla^{2}\bar{H}_{N+M}(\bs)\big)u)^{2}\leq\frac{B}{N+M},
\]
where we may need to increase the constant $B=B(\xi)$.

The bounds as in Part \ref{enu:pt1} of the lemma therefore follow
from the Borell-TIS inequality \cite{Borell,TIS}. \qed

\section{\label{sec:lemsubadd}Proof of Lemma \ref{lem:subadd}}

Recall the definition of the perturbed Hamiltonian
\begin{align}
\bar{H}_{N}^{x}(\bs) & =H_{N}(\bs)+s_{N}g_{N}^{x}(\bs),\label{eq:Hbar2}\\
g_{N}^{x}(\bs) & =\sum_{p=1}^{\infty}2^{-p}x_{p}g_{N,p}(\bs).\nonumber 
\end{align}
Here $x=(x_{p})_{p\geq1}$ are uniform variables in $[1,2]$, which
in the current proof we include in the notation to make the dependence
on $x$ explicit. Let $y=(y_{p})_{p\geq1}$ be an independent copy
of $x$. Define
\begin{equation}
\tilde{H}_{N,M}^{x,y}(\br,\btau)=H_{N+M}(\br,\btau)+s_{N}g_{N}^{x}(\br)+s_{M}g_{M}^{y}(\btau).\label{eq:Htilde}
\end{equation}

\begin{lem}
Let
\begin{align*}
A_{N,M} & =\E\log\int_{S_{N}\times S_{M}}e^{\bar{H}_{N+M}^{x}(\br,\btau)}d\mu_{N}\times\mu_{M}(\br,\btau)\\
 & -\E\log\int_{S_{N}\times S_{M}}e^{\tilde{H}_{N,M}^{x,y}(\br,\btau)}d\mu_{N}\times\mu_{M}(\br,\btau).
\end{align*}
Then, 
\[
\limsup_{M\to\infty}\limsup_{N\to\infty}\frac{1}{M}|A_{N,M}|=0.
\]
\end{lem}

\begin{proof}
For $(\br^{1},\btau^{1})$ and $(\br^{2},\btau^{2})$ in $S_{N}\times S_{M}$
denote
\begin{equation}
\begin{aligned}R_{1,2}^{1} & =\frac{1}{N}\br^{1}\cdot\br^{2},\qquad R_{1,2}^{2}=\frac{1}{M}\btau^{1}\cdot\btau^{2},\\
R_{1,2} & =\frac{N}{N+M}R_{1,2}^{1}+\frac{M}{N+M}R_{1,2}^{2}.
\end{aligned}
\label{eq:R}
\end{equation}
Then,
\[
\E\left(\bar{H}_{N+M}(\br^{1},\btau^{1})\bar{H}_{N+M}(\br^{2},\btau^{2})\right)=(N+M)\xi(R_{1,2})+\eta_{N+M}^{x}(R_{1,2})
\]
and
\begin{align*}
 & \E\left(\tilde{H}_{N,M}(\br^{1},\btau^{1})\tilde{H}_{N,M}(\br^{2},\btau^{2})\right)\\
 & =(N+M)\xi(R_{1,2})+\eta_{N}^{x}(R_{1,2}^{1})+\eta_{M}^{y}(R_{1,2}^{2}),
\end{align*}
where we define
\begin{equation}
\eta_{N}^{x}(t):=s_{N}^{2}\sum_{p=1}^{\infty}4^{-p}x_{p}^{2}t^{p}.\label{eq:eta}
\end{equation}

The difference of the two covariance functions above can be bounded
by 
\begin{equation}
\begin{aligned} & |\eta_{N+M}^{x}(R_{1,2})-\eta_{N}^{x}(R_{1,2}^{1})-\eta_{M}^{y}(R_{1,2}^{2})|\\
 & \leq|\eta_{N+M}^{x}(R_{1,2})-\eta_{N+M}^{x}(R_{1,2}^{1})|+|\eta_{N+M}^{x}(R_{1,2}^{1})-\eta_{N}^{x}(R_{1,2}^{1})|+|\eta_{M}^{y}(R_{1,2}^{2})|.
\end{aligned}
\label{eq:eta2}
\end{equation}
For any $x_{p},\,y_{p}\in[1,2]$ we have the following. The first
term in the right-hand side of (\ref{eq:eta2}) is bounded by
\[
\frac{d}{dt}\eta_{N+M}^{x}(1)\cdot|R_{1,2}-R_{1,2}^{1}|\leq\frac{2M}{N+M}s_{N+M}^{2}\sum_{p=1}^{\infty}4^{1-p}p.
\]
The middle term is bounded by
\[
(s_{N+M}^{2}-s_{N}^{2})\sum_{p=1}^{\infty}4^{1-p}.
\]
And the last term is bounded by
\[
s_{M}^{2}\sum_{p=1}^{\infty}4^{1-p}.
\]

For large $N$, the sum of all three above is bounded by 
\[
CM(N+M)^{2c-1}+CM^{2c},
\]
or some constant $C>0$, where $c\in(1/4,1/2)$ is the constant such
that $s_{N}=N^{c}$.

By an interpolation argument using Gaussian integration by parts,
this easily implies that 
\[
|A_{N,M}|\leq CM(N+M)^{2c-1}+CM^{2c},
\]
from which the lemma follows. Here we skip the details on Gaussian
integration by parts as this is a standard application and since it
will be used in a more complicated situation below where we give a
full explanation.
\end{proof}
To complete the proof of Lemma \ref{lem:subadd} we will show that
\begin{equation}
\E\log\int_{S_{N}\times S_{M}}e^{\tilde{H}_{N,M}^{x,y}(\br,\btau)}d\mu_{N}\times\mu_{M}(\br,\btau)\geq N\bar{F}_{N}+M\bar{F}_{M}-C_{N,M},\label{eq:A}
\end{equation}
for some $C_{N,M}>0$ as in the statement of the lemma. To prove this
we will now use the Guerra-Toninelli interpolation technique \cite{GuerraToninelli}.
Unlike the original argument of \cite{GuerraToninelli}, here we include
in the Hamiltonians the perturbation terms in order to be able to
invoke Talagrand's positivity principle at the end of the proof. The
latter will be required for mixtures including odd interactions or,
more precisely, mixtures such that $\xi$ is not on $[-1,1]$ (see
Footnote \ref{fn:evenmix} below).

Suppose that $H_{N+M}(\br,\btau)$, $H_{N}(\br)$, $H_{M}(\btau)$,
$g_{N}^{x}(\br)$ and $g_{N}^{y}(\br)$ are defined on the same probability
space such that they are all independent of each other, conditionally
and unconditionally on the uniform independent variables $x$ and
$y$. We will denote integration w.r.t. the randomness of the uniform
variables $x$ and $y$ by $\E_{u}$ and integration w.r.t. to all
Gaussian variables in the definition of the Hamiltonians by $\E_{g}$.

Define on $S_{N}\times S_{M}$ an interpolating Hamiltonian in $t\in[0,1]$,
\begin{align*}
H_{t}(\br,\btau)= & \sqrt{t}H_{N+M}(\br,\btau)\sqrt{1-t}\left(H_{N}(\br)+H_{M}(\btau)\right)\\
 & +s_{N}g_{N}^{x}(\br)+s_{M}g_{M}^{y}(\btau).
\end{align*}
Define the partition function 
\[
Z_{t}=\int_{S_{N}\times S_{M}}e^{H_{t}(\br,\btau)}d\mu_{N}\times\mu_{M}(\br,\btau)
\]
and free energy 
\[
\varphi(t)=\E_{u}\E_{g}\log Z_{t}.
\]
Let $G_{t}$ denote the corresponding Gibbs measure on $S_{N}\times S_{M}$
with density
\begin{equation}
\frac{dG_{t}}{d\mu_{N}\times\mu_{M}}(\br,\btau)=\frac{\exp H_{t}(\br,\btau)}{Z_{t}}.\label{eq:Gt}
\end{equation}

Note that $\varphi(0)=N\bar{F}_{N}+M\bar{F}_{M}$ and $\varphi(1)$
is equal to the left-hand side of (\ref{eq:A}). To complete the proof
of the lemma it remains to show that 
\begin{equation}
\varphi(1)-\varphi(0)\geq-C_{N,M},\label{eq:Ceps}
\end{equation}
for some $C_{N,M}$ as above.

Using Gaussian integration by parts, one has that (see e.g. the proof
of \cite[Lemma 1.1]{PanchenkoBook})
\begin{align*}
\varphi'(t)= & \frac{1}{2}\E_{u}\E_{g}\left\langle \frac{1}{\sqrt{t}}H_{N+M}(\br,\btau)-\frac{1}{\sqrt{1-t}}\left(H_{N}(\br)+H_{M}(\btau)\right)\right\rangle _{t}\\
= & -\frac{1}{2}\E_{u}\E_{g}\left\langle U_{N,M}\right\rangle _{t},
\end{align*}
where
\[
U_{N,M}:=(N+M)\left(\xi(R_{1,2})-\frac{N}{N+M}\xi(R_{1,2}^{1})-\frac{M}{N+M}\xi(R_{1,2}^{2})\right),
\]
$\langle\cdot\rangle_{t}$ denotes averaging of $(\br^{1},\btau^{1})$
and $(\br^{2},\btau^{2})$ with respect to $G_{t}^{\otimes2}$ and
we use the notation from (\ref{eq:R}). To prove (\ref{eq:Ceps})
we will show that\footnote{\label{fn:evenmix}Note that for if the mixture is an even function
$\xi(t)=\xi(-t)$, then $\xi(t)$ is convex on $[-1,1]$ and (\ref{eq:Dbd})
follows immediately. The rest of the proof deals with arbitrary $\xi(t)$
which are in general only convex on $[0,1].$}
\begin{equation}
\limsup_{M\to\infty}\limsup_{N\to\infty}\sup_{t\in[0,1]}\frac{1}{M}\E_{u}\E_{g}\left\langle U_{N,M}\right\rangle _{t}\leq0.\label{eq:Dbd}
\end{equation}

Define
\[
\hat{H}_{t}(\br)=\int H_{t}(\br,\btau)d\mu_{M}(\btau)
\]
and let 
\[
\frac{d\hat{G}_{t}}{d\mu_{N}}(\br)=\frac{\exp\hat{H}_{t}(\br)}{\int\exp\hat{H}_{t}(\br')d\mu_{N}}
\]
be the corresponding Gibbs measure. Clearly,
\[
G_{t}^{\otimes2}(R_{1,2}^{1}\in\cdot)=\hat{G}_{t}^{\otimes2}(R_{1,2}^{1}\in\cdot).
\]

Note that for any $t$, we may write 
\[
\hat{H}_{t}(\br)=\hat{H}_{t}^{y}(\br)+s_{N}g_{N}^{x}(\br),
\]
for some Hamiltonian $\hat{H}_{t}^{y}(\br)$ which is independent
of $g_{N}^{x}(\br)$. Hence, by Talagrand's positivity principle \cite{TalagBook2003},
see Theorem 3.4 in \cite{PanchenkoBook},
\begin{align*}
 & \lim_{N\to\infty}\sup_{M,t}\E_{u}\E_{g}G_{t}^{\otimes2}(R_{1,2}^{1}\leq-\epsilon_{N})\\
 & =\lim_{N\to\infty}\sup_{M,t}\E_{u}\E_{g}\hat{G}_{t}^{\otimes2}(R_{1,2}^{1}\leq-\epsilon_{N})=0,
\end{align*}
for some non-increasing sequence $\epsilon_{N}\to0$.

By a similar argument, (with the same sequence $\epsilon_{M}$)
\[
\lim_{M\to\infty}\sup_{N,t}\E_{u}\E_{g}G_{t}^{\otimes2}(R_{1,2}^{2}\leq-\epsilon_{M})=0.
\]
And thus, (with $a\wedge b$ denoting the minimum of $a$ and $b$)
\[
\limsup_{M\to\infty}\limsup_{N\to\infty}\sup_{t\in[0,1]}\E_{u}\E_{g}G_{t}^{\otimes2}(R_{1,2}^{1}\wedge R_{1,2}^{2}\leq-\epsilon_{M})=0.
\]

Note that for any choice of $R_{1,2}^{1},\,R_{1,2}^{2}\in[-1,1]$,
\begin{align*}
|U_{N,M}| & \leq M\xi(R_{1,2})+N|\xi(R_{1,2})-\xi(R_{1,2}^{1})|+M|\xi(R_{1,2}^{2})|\\
 & \leq2M(\xi(1)+\xi'(1)).
\end{align*}
Hence, 
\begin{equation}
\limsup_{M\to\infty}\limsup_{N\to\infty}\frac{1}{M}\sup_{t\in[0,1]}\E_{u}\E_{g}\left\langle U_{N,M}\cdot\indic\{R_{1,2}^{1}\wedge R_{1,2}^{2}\leq-\epsilon_{M}\}\right\rangle _{t}=0.\label{eq:Dbd1}
\end{equation}

Define 
\[
\begin{aligned}R_{+}^{1} & =R_{1,2}^{1}\vee0,\qquad R_{+}^{2}=R_{1,2}^{2}\vee0,\\
R_{+} & =\frac{N}{N+M}R_{+}^{1}+\frac{M}{N+M}R_{+}^{2},\\
R_{+}' & =\frac{N}{N+M}R_{1,2}^{1}+\frac{M}{N+M}R_{+}^{2},
\end{aligned}
\]
and
\[
U_{N,M}^{+}:=(N+M)\left(\xi(R_{+})-\frac{N}{N+M}\xi(R_{+}^{1})-\frac{M}{N+M}\xi(R_{+}^{2})\right).
\]
Write
\begin{align}
 & |U_{N,M}-U_{N,M}^{+}|\nonumber \\
 & \leq|(N+M)\xi(R_{+}')+M\xi(R_{+}^{2})-(N+M)\xi(R_{1,2})-M\xi(R_{1,2}^{2})|\label{eq:U+1}\\
 & +|(N+M)\xi(R_{+})-N\xi(R_{+}^{1})-(N+M)\xi(R_{+}')+N\xi(R_{1,2}^{1})|.\label{eq:U+2}
\end{align}

Suppose that $R_{1,2}^{1}\wedge R_{1,2}^{2}\geq-\epsilon_{M}$. Then
(\ref{eq:U+1}) is bounded by $2\epsilon_{M}M\xi'(1)$ and (\ref{eq:U+2})
is bounded by $\epsilon_{M}\max_{s,r}|\frac{d}{ds}h(s,r)|$ where
we define 
\[
h(s,r):=(N+M)\xi(\frac{N}{N+M}s+\frac{M}{N+M}r)-N\xi(s).
\]
Note that
\[
\Big|\frac{d}{ds}h(s,r)\Big|=\Big|N\xi'(\frac{N}{N+M}s+\frac{M}{N+M}r)-N\xi'(s)\Big|\leq2M\xi''(1).
\]
Hence, on the event that $R_{1,2}^{1}\wedge R_{1,2}^{2}\geq-\epsilon_{M}$,

\[
U_{N,M}\leq U_{N,M}^{+}+2\epsilon_{M}M(\xi'(1)+\xi''(1)).
\]

Lastly, since $\xi$ in convex on $[0,1]$,
\[
U_{N,M}^{+}\leq0.
\]

By combining the two inequalities we obtain that 
\[
\limsup_{M\to\infty}\limsup_{N\to\infty}\frac{1}{M}\sup_{t\in[0,1]}\E_{u}\E_{g}\left\langle D_{N,M}\cdot\indic\{R_{1,2}^{1}\wedge R_{1,2}^{2}\geq-\epsilon_{M}\}\right\rangle _{t}\leq0,
\]
which together with (\ref{eq:Dbd1}) proves (\ref{eq:Dbd}) and completes
the proof.\qed

\bibliographystyle{plain}
\bibliography{master}

\end{document}